\documentclass{article}

\usepackage[english]{babel}
\usepackage{amssymb,amsmath,amsthm, amsfonts,latexsym,mathtext} 

 \newtheorem{thm}{Theorem}[section]
 \newtheorem{cor}[thm]{Corollary}
 \newtheorem{lem}[thm]{Lemma}

 \numberwithin{equation}{section}
 
 \newcommand{\RM}{\mathbb{R}}
 \newcommand{\bv}{\mathbf{v}}
 \newcommand{\bx}{\mathbf{x}}

 \newcommand{\bn}{\mathbf{n}}
 \newcommand{\NM}{\mathbb{N}}
 \newcommand{\ZM}{\mathbb{Z}}
 \newcommand{\CM}{\mathbb{C}}

 \newcommand{\curl}{\operatorname{curl}}
 \newcommand{\ds}{\operatorname{ds}}
 
 \newcommand{\dlambda}{\operatorname{d\lambda}} 
 
 \newcommand{\sign}{\operatorname{sign}}
 \newcommand{\dx}{\operatorname{d\mathbf{x}}}

\begin{document}

\title{Associated Weber-Orr transform, Biot-Savart Law and explicit solution of 2D Stokes system in exterior of the disc.}


\maketitle

\begin{abstract}
In this article we derive the explicit solution of 2-D Stokes system in exterior of the disc with no-slip condition on inner boundary and given velocity $\bv_\infty$ at infinity. It turned out it is the first application of the associated Weber-Orr transform to mathematical physics in comparison to classical Weber-Orr transform which is used in many researches. From no-slip condition for velocity field we will obtain Robin-type boundary condition for vorticity. Then the initial-boundary value problem for vorticity will be solved with help of the associated Weber-Orr transform. Also the explicit formula of Biot-Savart Law in polar coordinates will be given.

Primary 76D07; Secondary 33C10.
\end{abstract}

\tableofcontents

\section*{Preliminary}
The exterior domain is the most natural area for study the Stokes flow past an obstacles. The advantage of the planar flow of the incompressible fluid is the fact that the motion equations  are reduced to only one vorticity equation. It works well in the case of Cauchy problem when the domain is the hole space $\RM^2$. But in exterior domains when the fluid interacts with the object by no-slip condition, then the Dirichlet boundary condition for velocity transforms to integral relations deduced from Biot-Savart law which describes the inverse of the curl operator and has rather complicated form. In this article we will show that for cylindrical domain $B_{r_0}=\{\bx \in \RM^2,~|\bx|\geq r_0 \},~r_0>0$ the no-slip boundary condition really can be transformed to another boundary condition for vorticity. 

The reasoning question is how the integral relation can be transformed to boundary ones since the Biot-Savart law doesn't admit integration by explicit formulas. The matter of fact that this boundary condition is given in terms of Fourier coefficients $w_k(t,r)$ of vorticity function $w(t,\bx)$  and these coefficients itself contain integration by polar angle. Since the Biot-Savart law is the convolution of the vortex fundamental solution with velocity, then the Fourier decomposition breaks the integral relations into series of Robin-type boundary conditions for $w_k(t, r)$ of the type
\begin{equation}\label{int:robin_bound}
r_0\frac{\partial w_k(t,r)}{\partial r}\Big|_{r=r_0} + |k| w_k(t,r_0) = 0,~k \in \ZM.
\end{equation}
And this relations define the no-slip boundary condition $\bv(t,\bx)\Big|_{|\bx|=r_0}=0$ in terms of vorticity function for Stokes flow with prescribed velocity $\bv_\infty$ at infinity. Here we suppose that initial data satisfies boundary conditions.
 
In this article we will find explicit formula for solution of the non-stationary Stokes system in exterior of the disc $B_{r_0}$ with given velocity value at infinity and no-slip  condition at the boundary. With help of vorticity operator we reduce Stokes system to heat equation with zero Robin-type boundary condition (\ref{int:robin_bound}) and zero boundary condition at infinity.
 
The initial boundary value problem for the heat equation in exterior of the disc with various types of boundary conditions could be solved with help of the Weber-Orr transform. In general form it can be written as
\begin{equation}\label{int:weberorr}
W_{k,l}[f](\lambda) = \int_{r_0}^\infty R_{k,l}(\lambda,s) f(s) s \ds,~k,l\in \RM
\end{equation}
where
$$
R_{k,l}(\lambda,s) = J_{k}(\lambda s)Y_{l}(\lambda r_0) - Y_{k}(\lambda s)J_{l}(\lambda r_0), 
$$
$J_k(r)$, $Y_k(r)$ - are the Bessel functions of the first and second type (see \cite{BE} \cite{W}).

\vskip 5pt
The inverse transform is defined by the formula

\begin{equation}\label{int:weberorrinv}
W^{-1}_{k,l}[\hat f](r) = \int_{0}^\infty \frac{R_{k,l}(\lambda,r)}{{J_{l}^2(\lambda r_0) + Y_{l}^2(\lambda r_0)}} \hat f (\lambda) \lambda \dlambda.
\end{equation}

Since $J_{k}$, $Y_{k}$  satisfy Bessel equation then the terms $e^{-\lambda^2 t + ik\varphi} J_{k}(\lambda r)$,\\ $e^{-\lambda^2 t + ik\varphi} Y_{k}(\lambda r)$ satisfy heat equation, and 
\begin{equation}\label{int:explsol}
w_k(t,r) = W^{-1}_{k,l} \left [ e^{-\lambda^2 t} W_{k,l} [w_k(0,r)](\lambda) \right ](t,r)
\end{equation}
multiplied by $e^{ik\varphi}$ do the same.

In order to satisfy the initial datum $w_k(0,r)$ at $t=0$ we need invertibility of $W_{k,l}$ and $W^{-1}_{k,l}$. When $l=k$ the above formulas are referred to classical Weber-Orr transform. Titchmarsh\cite{Tm}\cite{T} proved invertibility of this transform for functions $f(r)$ : $f(r) \sqrt r \in L_1(r_0,\infty)$ with bounded variations in the following form:
\begin{equation}\label{int:invidentity}
\frac {f(r+0) + f(r-0)}2 =  W^{-1}_{k,l} \left [W_{k,l} [f(\cdot)] \right ](r).
\end{equation}

For $l=k+1$ and non-negative $k\in\ZM$ the above formulas define associated Weber-Orr transform $W_{k,k+1}$ with inverse transform $W^{-1}_{k,k+1}$ satisfying the same identity above. This transform applies to elasticity theory and was the subject of some investigations \cite{S}\cite{KO}\cite{N}.  

Let's ask the question: what happens in the case $k>l$, $k \in \NM \cup \{0\}$ and what is the applications area of this transform? In present paper we will answer on this question when $l=k-1$. In this case the above formula (\ref{int:invidentity}) is incorrect. We will deduce correct formula of Weber-Orr transform in order to satisfy invertibility. And then we will present transparent application of this transform to fluid dynamics problems.

\vskip 5pt
So, when solving the Stokes system in exterior of the disc we will concentrate on the associated Weber-Orr transform $W_{k,k-1}$. Due to the fact that $R_{k,k-1}(\lambda,r)$ satisfies no-slip boundary condition (\ref{int:robin_bound}) for $k \in \NM \cup \{0\}$, and so $w_k(t,r)$ given by (\ref{int:explsol}) does the same, then the associated Weber-Orr transform $W_{k,k-1}$ will be the central point of our investigation. We can say that at this moment the only equations in which this transform finds its  application - are the Stokes and Navier-Stokes systems.  

The case of $l=k-1$ stands out from the rest Weber-Orr transforms since the formula (\ref{int:invidentity}) is no longer true and needs adding some corrections. The point is that the Robin-type condition 
$$\frac{\partial w_k(t,r)}{\partial r}\Big|_{r=r_0} + \alpha_k w_k(t,r_0) = 0$$ 
strongly depend on the sign of the $\alpha_k$, and the case $\alpha_k\geq0$ requires more thorough study.

Nasim\cite{N} investigated invertibility of $W_{k,k-1}$ and derived additional term for inverse transform $W^{-1}_{k,k-1}$. But his formula also needs some corrections to get the invertibility relation (\ref{int:invidentity}). We will find final formula of $W^{-1}_{k,k-1}$ in order to be inverse to $W_{k,k-1}$. As we will see, only in the case $k \geq 2$ the inverse transform must be supplied with additional correcting term, when for $k=0, 1$ the invertibility identity stays valid with $W^{-1}_{k,k-1}$ given by (\ref{int:weberorr}). For $k \geq 2$ the correcting term involves 
$$\int_{r_0}^\infty s^{-|k|+1}w_k(s)\ds,$$ which as we will see will be equal to zero when no-slip boundary condition is satisfied. 

And so, finally, we can say, that for both Stokes and Navier-Stokes systems with no-slip boundary condition the formula (\ref{int:weberorr}) presents actual inverse transform for $W_{k,k-1}$ and the solution of the heat equation is really provided by (\ref{int:explsol}). But in the common case the formulas (\ref{int:weberorr}), (\ref{int:weberorrinv}) do not satisfy the invertibility property (\ref{int:invidentity}).   

The structure of the paper is following. In the next section the statement of the initial boundary value problem is given. In section 2 we will deduce Biot-Savart law in polar coordinates for exterior of the disc. There will be considered both no-slip and slip boundary conditions. In section 3 we deduce no-slip boundary condition in terms of vorticity. And in section 4 we will find the solution of the Stokes system.

\section{The statement of the problem and the main result}
Consider the Stokes system defined in exterior of the disc  $B_{r_0}=\{\bx \in \RM^2,~|\bx| > r_0 \},~r_0>0$ with given horizontal flow at infinity $\bv_\infty = (v_{\infty},0) \in \RM^2$:
\begin{eqnarray}
&&\partial_t \bv - \Delta \bv = \nabla p \label{maineq}\\
&&{\rm div}~\bv(t,\bx)=0 \label{freediv}\\
&&\bv(0,\bx)=\bv_0(\bx) \label{init}\\
&&\bv(t,\bx)=0,~|\bx|=r_0 \label{bound}\\
&&\bv(t,\bx) \to \bv_\infty,~|\bx|\to \infty. \label{boundinf}
\end{eqnarray}
Here $\bv(t,\bx)=(v_1(t,\bx),v_2(t,\bx))$ is the velocity field and $p(t,\bx)$ is the pressure.

Applying to equations above the curl operator  
$w(t,\bx)=$ ${\rm curl}~\bv(t,\bx)$ $=\partial_{\bx_1}v_2 -  \partial_{\bx_2}v_1$ we get the system of equations for vorticity
\begin{eqnarray} 
\frac{\partial w(t,\bx)}{\partial t}	 - \Delta w = 0,  \label{maineqw} \\
w(0,\bx)=w_0(\bx) \label{initw}\\
\curl^{-1} w(t,\bx) \Big|_{|\bx|=r_0} = 0, \label{boundw}\\
w(t,\bx) \to 0,~|\bx|\to \infty \label{boundinfw}
\end{eqnarray} 
with initial datum $w_0(\bx)={\rm curl}~\bv_0(\bx)$. 

The last system will be the main object of our study instead of (\ref{maineq})-(\ref{boundinf}). 

%
%

It is naturally assumed, that the flow at infinity is the plane-parallel ones with zero circularity:
\begin{equation} \label{zerocirculation}
\lim_{R\to\infty}\oint_{|\bx|=R} \bv \cdot d\mathbf{l} = 0.
\end{equation} 
Further considerations will be held under zero-circulation assumption which due to Stokes formula and no-slip condition leads to zero mean vorticity:
\begin{eqnarray*}
\lim_{R\to\infty}\oint_{|\bx|=R} \bv(t,\bx) \cdot d\mathbf{l} = \int_{B_{r_0}} \curl \bv(t,\bx) \dx \nonumber \\  + 
\oint_{|\bx|=r_0} \bv(t,\bx) \cdot d\mathbf{l} = \int_{B_{r_0}} w(t,\bx) \dx = 0.
\end{eqnarray*} 

We will use Fourier expansions in polar coordinates $r$, $\varphi$: 
\begin{eqnarray}
 &&\bv(t,r,\varphi) = \sum_{k=-\infty}^\infty \bv_{k}(t,r)e^{ik\varphi},\nonumber \\
 &&\bv_0(r,\varphi) = \sum_{k=-\infty}^\infty \bv^0_{k}(r)e^{ik\varphi} \nonumber \\
 &&w(t,r,\varphi) = \sum_{k=-\infty}^\infty w_k(t,r)e^{ik\varphi},\nonumber \\
 &&w_0(r,\varphi) = \sum_{k=-\infty}^\infty w^0_{k}(r)e^{ik\varphi}. \nonumber
\end{eqnarray}
Here $\bv(t,r,\varphi) = (v_r, v_\varphi)$, $\bv_k(t,r) = (v_{r,k}, v_{\varphi,k}) =\frac 1{2\pi} \int_0^{2\pi} \bv(t,r,\varphi) e^{-ik\varphi} d\varphi $, $\bv_0(r,\varphi) = (v^0_{r}, v^0_{\varphi})$, $\bv^k_0(r) = (v^0_{r,k}, v^0_{\varphi,k})=\frac 1{2\pi} \int_0^{2\pi} \bv_0(r,\varphi) e^{-ik\varphi} d\varphi$ - are the vector fields decomposed on radial and tangent components. 

\vskip 5pt
The main result of the paper is
\begin{thm} Let vector field $\bv_0(\bx)$ satisfies (\ref{freediv}), (\ref{bound}), (\ref{boundinf}), (\ref{zerocirculation}), $\curl \bv_0(\bx)$ $ \in L_1(B_{r_0})$, and its Fourier series as well as Fourier series for vorticity converges and coefficients $w_k^0(r)$ satisfy $w_k^0(r) \sqrt r \in L_1(r_0,\infty)$, $k \in \ZM$. Then the solution of (\ref{maineqw})-(\ref{boundinfw}) is defined via Fourier coefficients:
$$
w_k(t,r) = W^{-1}_{|k|,|k|-1} \left [ e^{-\lambda^2 t} W_{|k|,|k|-1} [w^0_k(\cdot)](\lambda) \right ](t,r),
$$
where $W_{|k|,|k|-1}$, $W^{-1}_{|k|,|k|-1}$ are the associated Weber-Orr transforms (\ref{int:weberorr}), (\ref{int:weberorrinv}).
\end{thm}

\section{Biot-Savart law in polar coordinates.}

Now we study when the solenoidal velocity field $\bv(\bx)$ can be uniquely restored from its vorticity $w(\bx)$.
Consider the following system
\begin{eqnarray}
&&\rm{div}~ \bv(\bx) = 0 \label{freediv2} \\
&&\rm{curl}~ \bv(\bx) = w(\bx) \label{curleq} \\
&&\bv(\bx)=0,~|\bx|=r_0 \label{bound2}\\
&&\bv(\bx)\to\bv_\infty,~|\bx|\to \infty \label{boundinf2}.
\end{eqnarray}

Rewrite (\ref{freediv2}),(\ref{curleq}) in polar coordinates in terms of Fourier coefficients $v_{r,k}$, $v_{\varphi,k}$:
\begin{eqnarray*}
&&{\frac {1}{r}}{\frac {\partial }{\partial r}}\left(rv_{r,k}\right)+{\frac {ik}{r}} v_{\varphi,k} = 0,\\
&&{\frac {1}{r}}{\frac {\partial }{\partial r}}\left(rv_{\varphi,k}\right)-{\frac {ik}{r}} v_{r,k} = w_k.
\end{eqnarray*}

The basis for solutions of homogeneous system when $w_k \equiv 0$ consists of two vectors:
\begin{align*}
\begin{pmatrix}
v^1_{r,k} \\
v^1_{\varphi,k} 
\end{pmatrix} 
=
\begin{pmatrix}
ir^{-k-1} \\
r^{-k-1} 
\end{pmatrix}
, \\
\begin{pmatrix}
v^2_{r,k} \\
v^2_{\varphi,k} 
\end{pmatrix} 
=
\begin{pmatrix}
ir^{k-1} \\
-r^{k-1} 
\end{pmatrix}.
\end{align*}

Rewrite the horizontal flow $\bv_\infty = (v_\infty,0)$ 
in polar coordinates: 
$$
v_\infty\begin{pmatrix}
\cos\varphi \\
-\sin\varphi
\end{pmatrix}
=
v_\infty\begin{pmatrix}
e^{i\varphi}/2 + e^{-i\varphi}/2 \\
-e^{i\varphi}/2i + e^{-i\varphi}/2i
\end{pmatrix}.
$$

Suppose, that Fourier coefficients $w_k(s)$ of $w(\bx)$ satisfy:
\begin{align}\label{BS:cond}
w_k(s) \sqrt s \in L_1(r_0,\infty), k \in \ZM \setminus \{0\},  \\ 
w_0(s) s \in L_1(r_0,\infty). \nonumber
\end{align}

Using the Fourier expansion of $\bv_\infty$ the solution of the system (\ref{freediv2}), (\ref{curleq}), (\ref{boundinf2}) will be given by the following formula: 
\begin{flalign}
k>0:&& \label{BS:bck1} \\
&v_{r,k}(r) = \frac{ir^{-k-1}}2 \int_{r_0}^r s^{k+1}w_k(s)\ds + \frac{ir^{k-1}}2 \int_r^\infty s^{-k+1}w_k(s)\ds + \frac{\delta_{k,1}}2 v_\infty \nonumber \\ 
&+ \alpha_k i r^{-k-1}, \nonumber\\
&v_{\varphi,k}(r) = \frac{r^{-k-1}}2 \int_{r_0}^r s^{k+1}w_k(s)\ds - \frac{r^{k-1}}2 \int_r^\infty s^{-k+1}w_k(s)\ds - \frac{\delta_{k,1}}{2i} v_\infty \nonumber\\
& + \alpha_k r^{-k-1},\nonumber 
\end{flalign}
\begin{flalign}
k<0: &&\label{BS:bck2}  \\
&v_{r,k} = - \frac{ir^{k-1}}2 \int_{r_0}^r s^{-k+1}w_k(s)\ds - \frac{ir^{-k-1}}2 \int_r^\infty s^{k+1}w_k(s)\ds + \frac{\delta_{k,-1}}2 v_\infty \nonumber\\ 
&+ \alpha_k i r^{k-1}, \nonumber\\
&v_{\varphi,k} = \frac{r^{k-1}}2 \int_{r_0}^r s^{-k+1}w_k(s)\ds - \frac{r^{-k-1}}2 \int_r^\infty s^{k+1}w_k(s)\ds + \frac{\delta_{k,-1}}{2i} v_\infty \nonumber\\ 
&- \alpha_k r^{k-1}, \nonumber  
\end{flalign}
\begin{flalign}
k=0: \label{BS:bck3} &&\\
&v_{r,0} = \frac {\alpha_0}r, \nonumber\\
&v_{\varphi,0} = \frac 1r \int_{r_0}^r s w_0(s)\ds + \frac {\theta}r, \nonumber
\end{flalign}
where $\delta_{k,1}$, $\delta_{k,-1}$ are the Kronecker deltas, $\{\alpha_k\}_{k=-\infty}^\infty$, $\theta$ will be found from the rest boundary condition (\ref{bound2}). 

The term $\frac {\theta}r$ in the last formula describes discrete vortex motion 
\begin{equation} \label{discretevortex}
\sigma_{\theta} (\bx) =  \theta \frac {\bx^\perp } { |\bx|^2},
\end{equation}
where $~\bx^\perp = (-x_2,x_1)$, and $\theta$ defines its intensity.

\subsection{No-slip condition}
For no-slip condition (\ref{bound2}) from $v_{r,k}(r_0)=v_{\varphi,k}(r_0)=0$ in virtue of (\ref{BS:bck3}) immediately follows $\alpha_0 = 0$, and $\theta = 0$. For $k>0$ from (\ref{BS:bck1}) follows
\begin{flalign*}
&v_{r,k}(r_0) = \frac{ir_0^{k-1}}2 \int_{r_0}^\infty s^{-k+1}w_k(s)\ds + \frac{\delta_{k,1}}2 v_\infty + \alpha_k i r^{-k-1} = 0, \nonumber\\
&v_{\varphi,k}(r_0) = - \frac{r_0^{k-1}}2 \int_{r_0}^\infty s^{-k+1}w_k(s)\ds - \frac{\delta_{k,1}}{2i} v_\infty 
 + \alpha_k r^{-k-1} = 0,\nonumber 
\end{flalign*}
and so $\alpha_k = 0$, $k \in \NM$. In similar way from (\ref{BS:bck2}) we have $\alpha_k = 0$ for negative $k$.

Then if there exists the solution of (\ref{freediv2}) - (\ref{boundinf2}) satisfying (\ref{BS:cond}), then it will be given by (\ref{BS:bck1})-(\ref{BS:bck3}) with $\alpha_k = 0$, $k \in \ZM$, and $\theta = 0$. 

Under assumption of zero-circularity (\ref{zerocirculation}) from Stokes' theorem we have 
$$
\lim_{R\to\infty}\oint_{|\bx|=R} \bv(\bx) \cdot d\mathbf{l} = \int_{B_{r_0}} w(\bx) \dx = 2\pi \int_{r_0}^\infty  w_0(s) s \ds = 0,
$$
and the term in (\ref{BS:bck3}) is represented as
$$
\frac 1r \int_{r_0}^r s w_0(s)\ds = \frac 1{2r} \int_{r_0}^r s w_0(s)\ds - \frac 1{2r} \int_r^\infty s w_0(s)\ds.
$$



Finally, the above formulas, which constitute Biot-Savart law in cylindrical domains, we rewrite for $k \in \ZM$:
\begin{eqnarray} \label{BiotSavar1}
&&v_{r,k} = \sign(k)  \frac{ir^{-|k|-1}}2 \int_{r_0}^r s^{|k|+1}w_k(s)\ds \\ 
&&~~~~~~~~~~~~~~~~~~+ \sign(k) \frac{ir^{|k|-1}}2 \int_r^\infty s^{-|k|+1}w_k(s)\ds  + \frac{\delta_{|k|,1}}2 v_\infty \nonumber \\ \label{BiotSavar2}
&&v_{\varphi,k} = \frac{r^{-|k|-1}}2 \int_{r_0}^r s^{|k|+1}w_k(s)\ds \nonumber \\ 
&&~~~~~~~~~~~~~~~~~~- \frac{r^{|k|-1}}2 \int_r^\infty s^{-|k|+1}w_k(s)\ds - \sign(k) \frac{\delta_{|k|,1}}{2i} v_\infty,~~
\end{eqnarray}
where $\sign(k) = \begin{cases} 1,~k > 0 \\ 0,~k=0 \\ -1,~k<0 \end{cases}$ - signum function, $\delta_{|k|,1}$ is the Kronecker delta, $v_\infty$ - the velocity of horizontal flow at infinity. 

Formulas (\ref{BiotSavar1}), (\ref{BiotSavar2}) combined with (\ref{bound2}) lead to relations on vorticity ($k\in \ZM$): 
\begin{equation} \label{noslipcondintegral}
\int_{r_0}^\infty s^{-|k|+1}w_k(s)\ds = i v_\infty \delta_{|k|,1} \sign(k)  / r_0^{|k|-1}. 
\end{equation}

\begin{thm}[Biot-Savart Law in polar coordinates]\label{Biotpolarnoslip}
Let Fourier coefficients $w_k(r)$ satisfy (\ref{BS:cond}), (\ref{noslipcondintegral}). Then there exists the unique solution of (\ref{freediv2}) - (\ref{boundinf2}) given by (\ref{BiotSavar1}), (\ref{BiotSavar2}) 
with Fourier coefficients $v_{r,k}$, $v_{\varphi,k} \in L_\infty(r_0,\infty)$.
%
\end{thm}
\begin{proof}
Indeed, (\ref{bound2}), (\ref{boundinf2}) and the inclusion $v_{r,k}$, $v_{\varphi,k} \in L_\infty(r_0,\infty)$ follows from (\ref{BS:cond}), (\ref{noslipcondintegral}).  Formulas (\ref{BiotSavar1}), (\ref{BiotSavar2}) satisfy (\ref{freediv2}), (\ref{curleq}). Since the coefficients $\alpha_k = 0$, $k \in \ZM$, and $\theta = 0$ in (\ref{BS:bck1})-(\ref{BS:bck3}) are uniquely defined, then the solution is also unique. 
\end{proof}


\subsection{Slip condition}
Now we will weaken (\ref{bound2}) by
\begin{equation}\label{BS:slip}
\left(\bv(\bx), \bn \right) = 0,~|\bx|=r_0,
\end{equation}
where $\bn$ is the outer normal to the boundary and find the solution of the undetermined system (\ref{freediv2}), (\ref{curleq}), (\ref{boundinf2}) with (\ref{BS:slip}). It will obtain one degree of freedom which concerns to circularity of the flow. 

From (\ref{BS:slip}) follows $v_{r,k}(r_0)=0$, $k \in \ZM$. Then from (\ref{BS:bck1})-(\ref{BS:bck3})
\begin{align*}
&\alpha_k = - \frac{\sign(k) r_0^{2|k|}}2 \int_{r_0}^\infty s^{-|k|+1}w_k(s)\ds - \frac{r_0^{|k|+1}\delta_{|k|,1}}{2i} v_\infty,~k\in\ZM \setminus \{0\},\\
&\alpha_k=0,~k=0.
\end{align*}

So, the system (\ref{freediv2}), (\ref{curleq}), (\ref{BS:slip}), (\ref{boundinf2}) is not uniquely solvable due to arbitrary $\theta \in \RM$ in (\ref{BS:bck3}) which describes the intensity of the circulation around the origin. The vector field $\bv(\bx)$ is restored from its vorticity $w(\bx)$ up to circulation of $\bv(\bx)$ around boundary. In cartesian coordinates it means that to any solution of (\ref{freediv2}),(\ref{curleq}), (\ref{BS:slip}), (\ref{boundinf2}) $\bv(t,\bx)$ we can add discrete vortex (\ref{discretevortex}) and $\bv(t,\bx) + \sigma_{\theta} (\bx)$ will also be the solution of the same system  with any $\theta \in \RM$. In general, in exterior domains in contrast to simple connected ones, the vector field $\bv(\bx)$ is restored from its vorticity $w(\bx)$ with one degree of freedom - circularity around boundary (for more details see \cite{G2}).

\section{No-slip condition for vorticity in polar coordinates.}
Now we a ready to derive no-slip boundary condition for vorticity instead of integral relations (\ref{noslipcondintegral}).


Initial value problem (\ref{maineqw}), (\ref{initw}) in polar coordinates is given by equations
\begin{equation} \label{omegaeqpolar}
\frac{\partial w_k(t,r)}{\partial t}	 - \Delta_k w(t,r) = 0, ~w_k(0,r) = w^0_k(r),
\end{equation} 
where
$$
\Delta_k w_k(t,r) = \frac 1r \frac {\partial}{\partial r}\left(r \frac {\partial}{\partial r}w_k(t,r)\right) - \frac{k^2}{r^2} w_k(t,r).
$$

We supply (\ref{omegaeqpolar}) with Robin-type boundary condition:
\begin{equation}\label{noslip:robin_bound}
r_0\frac{\partial w_k(t,r)}{\partial r}\Big|_{r=r_0} + |k| w_k(t,r_0) = 0,~k \in \ZM.
\end{equation}

\begin{lem}The set
$$
\Omega = \{w_k^0(s) \in L_1(r_0,\infty,s^{-|k|+1}\ds) ~|~  \int_{r_0}^\infty s^{-|k|+1}w_k^0(s)\ds = const \}
$$ 
is invariant under the flow $e^{\Delta_k t}$ corresponding to (\ref{omegaeqpolar}) iff (\ref{noslip:robin_bound}) holds.
\end{lem}

\begin{proof} 
\begin{align*}
\frac d{dt}\int_{r_0}^\infty s^{-|k|+1}w_k(t,s)\ds = \int_{r_0}^\infty s^{-|k|+1} \Delta_k w_k(t,s)\ds \\
= \int_{r_0}^\infty s^{-|k|} \left (  \frac {\partial}{\partial s}\left(s \frac {\partial}{\partial s}w_k(t,s)\right) - \frac{k^2}{s}  w_k(t,s) \right ) \ds \\ 
= - r_0^{-|k|+1}\frac{\partial w_k(t,r)}{\partial r}\Big|_{r=r_0} + \int_{r_0}^\infty s^{-|k|} \left (  |k|  \frac {\partial}{\partial s}w_k(t,s) - \frac{k^2}{s} w_k(t,s) \right ) \ds\\
\\
= - r_0^{-|k|}\left(r_0 \frac{\partial w_k(t,r)}{\partial r}\Big|_{r=r_0} + |k| w_k(t,r_0) \right )
+ \int_{r_0}^\infty s^{-|k|-1}   \\
\times \left (  k^2 w_k(t,s) - k^2 w_k(t,s) \right ) \ds = 0
\end{align*}

\end{proof}

\begin{thm} \label{thmrobincond} Let initial vector field $\bv_0(\bx)$  satisfies(\ref{freediv}), (\ref{bound}), (\ref{boundinf}), (\ref{zerocirculation}), \\ $\curl \bv_0(\bx) \in L_1(B_{r_0})$, and its Fourier series as well as ones for vorticity converges and coefficients $w_k^0(r)$ satisfy $w_k^0(r) \sqrt r \in L_1(r_0,\infty)$, $k \in \ZM$. Then coefficients $w_k(t,r)$ of (\ref{maineqw})-(\ref{boundinfw}) satisfy Robin-type condition (\ref{noslip:robin_bound}).
\end{thm}

\begin{proof}From $\curl \bv_0(\bx)$ $ \in L_1(B_{r_0})$ follows
$w_0^0(r) r \in L_1(r_0,\infty)$ and (\ref{BS:cond}) is satisfied for $k = 0$.
No-slip condition imposes the relations (\ref{noslipcondintegral}), which must be satisfied for all $t>0$:
$$
\int_{r_0}^\infty s^{-|k|+1}w_k(t,s)\ds = i v_\infty \delta_{|k|,1} \sign(k)  / r_0^{|k|-1}, 
$$
i.e. $\Omega$ must be invariant under the flow $e^{\Delta_k t}$ corresponding to (\ref{omegaeqpolar}), (\ref{noslip:robin_bound}). Then the statement of the theorem is the direct result of the previous lemma.
\end{proof}

Thus, the problem (\ref{maineqw})-(\ref{boundinfw}) is reduced to (\ref{omegaeqpolar}), (\ref{noslip:robin_bound}). Now we will prove, that, as it was mentioned in preliminary, the relation  
$$\int_{r_0}^\infty s^{-|k|+1}w_k(s)\ds  = 0,~|k| \geq 2,$$ which follows from (\ref{noslipcondintegral}), will give the invertibility formula (\ref{int:invidentity}) for the  associated Weber-Orr transform $W_{k,k-1}$.  

\section{The solution of the Heat equation with Robin-type boundary condition. Derivation of the Weber-Orr transform.}

In this section we will solve the initial boundary value problem (\ref{omegaeqpolar}), (\ref{noslip:robin_bound}) reducing it to elliptic system with parameter (see \cite{AV}).

Laplace transform
$$
 \hat \omega(\tau,\bx)=\int_0^\infty e^{-\tau t}w(t,\bx) \dx
$$
reduces (\ref{maineqw}), (\ref{initw}) to Poison equation with parameter $\tau$
\begin{equation}\label{poison}
\Delta \hat \omega - \tau \hat \omega  = - w_0(\bx).
\end{equation}

Then Fourier coefficients $\hat \omega_k(\tau,r)$ satisfy differential equation
\begin{equation} \label{besselrhs}
 \frac{\partial^2 \hat \omega_k}{\partial r^2} + \frac1r \frac{\partial \hat \omega_k}{\partial r} - \Big(\frac{k^2}{r^2} + \tau \Big) \hat \omega_k = -w_k^0(r).
\end{equation}
 
Since the above equation is invariant under the change $k$ to $-k$, then we will carry out further considerations only for non-negative integer $k$. 

Basis of solutions to homogeneous Bessel equation (\ref{besselrhs}) when right-hand side $-w_k^0(r)$ is equal to zero consists of modified Bessel functions of the first and second kind $I_k(\sqrt{\tau}r)$, $K_k(\sqrt{\tau}r)$  (see \cite{BE}, \cite{W}). Following by \cite{G2} we exclude exponentially growing function $I_k(\sqrt{\tau}r)$ from the solution of (\ref{besselrhs}). The Wronskian of these functions is equal to:
\begin{equation}\label{wronsk}
W\left[I_k(r),K_k (r)\right]=-r^{-1}.
\end{equation}
Then 
$$K_k(\sqrt{\tau}r)\int_{r_0}^r w_k^0(s) I_k(\sqrt{\tau}s) s ds \nonumber \\ + 
I_k(\sqrt{\tau}r)\int_r^{\infty} w_k^0(s) K_k(\sqrt{\tau}s) s ds $$
will be the solution of non-homogeneous Bessel equation (\ref{besselrhs}). And finally we have the formula: 
\begin{eqnarray} \label{directsol}
 \hat \omega_k(\tau,r) = h_k(\tau) K_k(\sqrt{\tau}r) + K_k(\sqrt{\tau}r)\int_{r_0}^r w_k^0(s) I_k(\sqrt{\tau}s) s ds \nonumber \\ + 
I_k(\sqrt{\tau}r)\int_r^{\infty} w_k^0(s) K_k(\sqrt{\tau}s) s ds ,
\end{eqnarray}
where $h_k(\tau)$ will be founded from boundary condition (\ref{noslip:robin_bound}). To do this we apply the operator of boundary condition $$Z_{k,r_0}[\cdot]= r_0{\frac{\partial \cdot}{\partial r }} \Big |_{r=r_0} + k \cdot$$ to the expression 
$$f_k(\tau,r) := K_k(\sqrt{\tau}r)\int_{r_0}^r w_k^0(s) I_k(\sqrt{\tau}s) s ds  + 
I_k(\sqrt{\tau}r)\int_r^{\infty} w_k^0(s) K_k(\sqrt{\tau}s) s ds.$$

Since
$$
\frac{\partial f_k}{\partial r}(\tau, r_0) = \sqrt \tau I_k'(\sqrt \tau r_0) \int_{r_0}^\infty w_k^0(s) K_k(\sqrt{\tau}s) s ds,
$$
then
\begin{align*}
Z_{k,r_0} [f_k(\tau,r)] =  \sqrt \tau r_0 I_k'(\sqrt \tau r_0) \int_{r_0}^\infty w_k^0(s) K_k(\sqrt{\tau}s) s ds + k I_k(\sqrt{\tau}r_0)
 \nonumber \\ \times \int_{r_0}^\infty w_k^0(s) K_k(\sqrt{\tau}s) s ds.
\end{align*}

Using the differentiation relations on Bessel functions 
\begin{eqnarray*}
I_{k-1}(r)-I_{k+1}(r)={\frac {2k }{r}}I_{k }(r)\\
I_{k-1}(r)+I_{k+1}(r)=2{\frac {dI_{k }}{dr}}(r) \\
\end{eqnarray*}
we will have
$$
\sqrt \tau r_0 I_k'(\sqrt \tau r_0) + k I_k(\sqrt{\tau}r_0) = \sqrt \tau r_0 I_{k-1}(\sqrt \tau r_0).
$$

Then
\begin{align*}
Z_{k,r_0}[f_k(\tau,r)] =  r_0 \sqrt \tau  I_{k-1}(\sqrt \tau r_0) \int_{r_0}^\infty w_k^0(s) K_k(\sqrt{\tau}s) s ds.
\end{align*}
From
\begin{align*}
K_{k-1}(r)-K_{k+1}(r)=-{\frac {2k }{r}}K_{k }(r)\\
K_{k-1}(r)+K_{k+1}(r)=-2{\frac {dK_{k }}{dr}}(r)
\end{align*}
follows
\begin{align*}
Z_{k,r_0}[K_k(\sqrt \tau r)] =  -r_0 \sqrt \tau K_{k-1}(\sqrt \tau r_0).
\end{align*}

We have got the solution of (\ref{besselrhs}) with boundary (\ref{noslip:robin_bound}) for any non-negative $k \in \ZM$
\begin{align*} 
 \hat \omega_k(\tau,r) = \frac{K_{k}(\sqrt{\tau}r) I_{k-1}(\sqrt \tau r_0)}{K_{k-1}(\sqrt \tau r_0)} \int_{r_0}^\infty w_k^0(s) K_{k}(\sqrt{\tau}s) s ds \nonumber \\+ K_{k}(\sqrt{\tau}r)\int_{r_0}^r w_k^0(s) I_{k}(\sqrt{\tau}s) s ds  + 
I_{k}(\sqrt{\tau}r)\int_r^{\infty} w_k^0(s) K_{k}(\sqrt{\tau}s) s ds. 
\end{align*}
For negative $k \in \ZM$ previous formula stays valid when $k$ changes to $|k|$ in Bessel functions $K_{k}$, $K_{k-1}$, $I_{k}$, $I_{k-1}$. Finally we have
\begin{align}  \label{wksol}
 \hat \omega_k(\tau,r) = \frac{K_{|k|}(\sqrt{\tau}r) I_{|k|-1}(\sqrt \tau r_0)}{K_{|k|-1}(\sqrt \tau r_0)} \int_{r_0}^\infty w_k^0(s) K_{|k|}(\sqrt{\tau}s) s ds \nonumber \\+ K_{|k|}(\sqrt{\tau}r)\int_{r_0}^r w_k^0(s) I_{|k|}(\sqrt{\tau}s) s ds  + 
I_{|k|}(\sqrt{\tau}r)\int_r^{\infty} w_k^0(s) K_{|k|}(\sqrt{\tau}s) s ds. 
\end{align}

In order to obtain the solution $w(t,x)$ of the heat equation we apply inverse Laplace transform
$$
w(t,r,\varphi) = \frac1{2\pi i} \int_{\Gamma_\eta}e^{\tau t}\hat \omega(\tau,r,\varphi)d\tau,
$$
where $\Gamma_\eta = \{ \tau \in \CM, ~\operatorname{Re}\tau = \eta \}$, $\eta$ - fixed arbitrary positive real number.

Since $\sqrt \tau$ involved in (\ref{wksol}) is a branching function, then using Cauchy formula we change contour of integration from $\Gamma_\eta$ to $\Gamma_{-\eta, \varepsilon}^-$ with some $\varepsilon>0$, where
\begin{eqnarray}
\Gamma_{-\eta, \varepsilon}^- = ( -\eta -i\infty, -\eta -i\varepsilon] \cup
[-\eta - i\varepsilon, - i\varepsilon]     
\cup [- i\varepsilon, i\varepsilon] \nonumber \\
\cup [i\varepsilon, -\eta + i\varepsilon] 
\cup [-\eta + i\varepsilon, -\eta + i\infty). \nonumber
\end{eqnarray}

Define oriented contour $\gamma_{\pm,\eta} = [-\eta,0] \cup [0,-\eta]$. Then, going $\varepsilon \to 0$, $\eta \to \infty$ we will have  equalities:
\begin{align} 
&w(t,r,\varphi) = 
\frac1{2\pi i} \sum_{k=-\infty}^\infty e^{ik\varphi}  \int_{-\infty}^0 e^{\tau t}\hat \omega_k(\tau-i0,r)d\tau \nonumber \\  
&+ \frac1{2\pi i}  \sum_{k=-\infty}^\infty e^{ik\varphi}   \int_0^{-\infty} e^{\tau t}\hat \omega_k(\tau+i0,r)d\tau 
+ \sum_{k=-\infty}^\infty e^{ik\varphi} \operatorname{res}\limits_{\tau=0}[\hat \omega_k(\tau,r)] ,  \nonumber
\end{align}
where 
$$
{\mathrm  {res}}_{\tau=0}\,[\hat \omega_k(\tau,r)]=\lim _{{\rho \to 0}}{1 \over {2\pi i}}\int \limits _{{|\tau|=\rho }}\!\hat \omega_k(\tau,r)\,d\tau.
$$ 
 
Set $\tau = -\lambda^2$. Then, we will have  
\begin{eqnarray} \label{wformula}
w(t,r,\varphi) = \nonumber 
\frac1{\pi i} \sum_{k=-\infty}^\infty e^{ik\varphi} \int_0^\infty e^{-\lambda^2 t} \left ( w_k(-\lambda^2-i0,r)  - w_k(-\lambda^2+i0,r) \right )
 \lambda \mathrm{d}\lambda  \nonumber \\ + \sum_{k=-\infty}^\infty e^{ik\varphi} \operatorname{res}\limits_{\tau=0}[\hat \omega_k(\tau,r)].~~~
\end{eqnarray} 

We will use the following lemma, involving Bessel function of the first kind $J_k(r)$, which was proved in \cite{G2}:
\begin{lem} \label{besselrelationlem}
For $\lambda,r,s > 0$ modified Bessel functions $I_k,K_k$ will satisfy: 
$$
I_k(-i\lambda r)K_k(-i\lambda s) - I_k(i\lambda r)K_k(i\lambda s) = \pi i J_k(\lambda r)J_k(\lambda s).
$$
\end{lem}

Recall the definition of Hankel functions $H_k^{(1)}$, $H_{-k}^{(2)}$ (see \cite{BE}, \cite{W}):
\begin{align} \label{hankelfunctions}
&H_k^{(1)}(\lambda r) = J_k(\lambda r) + iY_k(\lambda r) \nonumber \\
&H_{-k}^{(2)}(\lambda r) =
(-1)^k(J_{k}(\lambda r) - iY_{k}(\lambda r)).
\end{align}

Then analytical continuation of $I_k,K_k$ on imaginary line is given as follows(see \cite{W}): 
\begin{align} \label{besselcontnuation}
&I_k(-i\lambda r)=(-i)^k J_k(\lambda r),\nonumber  \\  
&I_k(i\lambda r)=i^k J_k(\lambda r), \nonumber  \\ 
&K_k(-i\lambda r) = \frac\pi 2 i^{k+1}H_k^{(1)}(\lambda r), \nonumber  \\ 
&K_k(i\lambda r) = -\frac\pi 2 i^{k+1}H_{-k}^{(2)}(\lambda r). 
\end{align}

We decompose function $\hat \omega_k(\tau,r)$ in (\ref{wksol}) for $k\in\NM\cup\{0\}$ (as usual for negative integer $k$ the indexes in all Bessel functions involving $k$ should be changed to $|k|$):
$$
 \hat \omega_k(\tau,r) = G_{k,1}(\tau,r)+ G_{k,2}(\tau,r), 
$$ 
where $G_{k,1}(\tau,r)$, $G_{k,2}(\tau,r)$ are defined as 
\begin{align*}
G_{k,1}(\tau,r) &= \frac{K_k(\sqrt{\tau}r) I_{k-1}(\sqrt \tau r_0)}{K_{k-1}(\sqrt \tau r_0)} \int_{r_0}^\infty w_k^0(s) K_k(\sqrt{\tau}s) s ds, \nonumber \\
G_{k,2}(\tau,r) &= K_k(\sqrt{\tau}r)\int_{r_0}^r w_k^0(s) I_k(\sqrt{\tau}s) s ds \nonumber \\ 
&+ I_k(\sqrt{\tau}r)\int_r^{\infty} w_k^0(s) K_k(\sqrt{\tau}s) sds. \nonumber
\end{align*}

Then with help of (\ref{hankelfunctions}), (\ref{besselcontnuation}) we will have
\begin{align*}
&G_{k,1}(-\lambda^2-i0,r)  - G_{k,1}(-\lambda^2+i0,r) = \\
&=\int_{r_0}^\infty \Big ( \frac{K_k(-i\lambda r) I_{k-1}(-i\lambda r_0)K_k(-i\lambda s)}{K_{k-1}(-i\lambda r_0)} \\
&- \frac{K_k(i\lambda r) I_{k-1}(i\lambda r_0)K_k(i\lambda s)}{K_{k-1}(i\lambda r_0)} \Big ) w_k^0(s)  s ds \nonumber \\
&=\frac{\pi i}2 \int_{r_0}^\infty \Big ( - \frac{H_k^{(1)}(\lambda r) J_{k-1}(\lambda r_0)H_k^{(1)}(\lambda s)}{H_{k-1}^{(1)}(\lambda r_0)} \\
&+ \frac{H_{-k}^{(2)}(\lambda r) J_{k-1}(\lambda r_0)H_{-k}^{(2)}(\lambda s)}{H_{-(k-1)}^{(2)}(\lambda r_0)} \Big ) w_k^0(s)  s ds \nonumber \\
&=-\frac{\pi i}2 \int_{r_0}^\infty  \Big ( \frac{(J_k(\lambda r)+iY_k(\lambda r)) J_{k-1}(\lambda r_0)(J_k(\lambda s)+iY_k(\lambda s))}{J_{k-1}(\lambda r_0)+iY_{k-1}(\lambda r_0)} \\ 
~~~~& + \frac{(J_k(\lambda r)-iY_k(\lambda r)) J_{k-1}(\lambda r_0)(J_k(\lambda s)-iY_k(\lambda s))}{J_{k-1}(\lambda r_0)-iY_{k-1}(\lambda r_0)}  \Big ) w_k^0(s)  s ds.
\nonumber
\end{align*}

From lemma \ref{besselrelationlem} follows
\begin{align*}
G_{k,2}(-\lambda^2-i0,r)  - G_{k,2}(-\lambda^2+i0,r) = \pi i \int_{r_0}^\infty J_k(\lambda r) J_k(\lambda s) w_k^0(s) s ds,
\end{align*}
and
\begin{align*}&\hat \omega_k(-\lambda^2-i0,r) - \hat \omega_k(-\lambda^2+i0,r) \nonumber \\ 
&=G_{k,1}(-\lambda^2-i0,r)  - G_{k,1}(-\lambda^2+i0,r) \nonumber \\
&+ G_{k,2}(-\lambda^2-i0,r)  - G_{k,2}(-\lambda^2+i0,r)   
=-\frac{\pi i}2 \int_{r_0}^\infty \Big ( \\ & \frac{J_{k-1}(\lambda r_0)(J_k(\lambda r)+iY_k(\lambda r)) (J_k(\lambda s)+iY_k(\lambda s)) (J_{k-1}(\lambda r_0)-iY_{k-1}(\lambda r_0) )}{J_{k-1}(\lambda r_0)^2+Y_{k-1}(\lambda r_0)^2}\\
&+ \frac{J_{k-1}(\lambda r_0)(J_k(\lambda r)-iY_k(\lambda r)) (J_k(\lambda s)-iY_k(\lambda s))(J_{k-1}(\lambda r_0)+iY_{k-1}(\lambda r_0))}{J_{k-1}(\lambda r_0)^2+Y_{k-1}(\lambda r_0)^2}
\\
&-2\frac {J_k(\lambda r) J_k(\lambda s) (J_{k-1}(\lambda r_0)^2+Y_{k-1}(\lambda r_0)^2)} {J_{k-1}(\lambda r_0)^2+Y_{k-1}(\lambda r_0)^2} \Big )  w_k^0(s) s ds.\\
\end{align*}

After a series of transformations we will have
\begin{align*}&\hat \omega_k(-\lambda^2-i0,r) - \hat \omega_k(-\lambda^2+i0,r)\\
=&\pi i \int_{r_0}^\infty  \frac{ \left (J_{k-1}(\lambda r_0) Y_k(\lambda r) - Y_{k-1}(\lambda r_0) J_k(\lambda r) \right )
 } {J_{k-1}(\lambda r_0)^2 + Y_{k-1}(\lambda r_0)^2} \\
&\times \left ( J_{k-1}(\lambda r_0) Y_k(\lambda s) - Y_{k-1}(\lambda r_0) J_k(\lambda s) \right )  w_k^0(s) s ds  \\
&= \pi i \int_{r_0}^\infty  \frac{ R_{k,k-1}(\lambda, r) R_{k,k-1}(\lambda, s)  w_k^0(s) s ds} {J_{k-1}(\lambda r_0)^2 + Y_{k-1}(\lambda r_0)^2}
\end{align*}

Recall that we assumed $k \geq 0$. Then  due to invariance of Bessel equation (\ref{besselrhs}) under the change $k$ to $-k$ from (\ref{wformula}) we will find the solution
\begin{align*}
&w(t,r,\varphi) = \nonumber 
\sum_{k=-\infty}^\infty e^{ik\varphi} \int_0^\infty  \frac{  R_{|k|,|k|-1}(\lambda, r)} 
{J_{|k|-1}(\lambda r_0)^2 + Y_{|k|-1}(\lambda r_0)^2} \\ 
& \times \left ( \int_{r_0}^\infty R_{|k|,|k|-1}(\lambda, s) w_k^0(s) s ds \right )  e^{-\lambda^2 t} \lambda \mathrm{d}\lambda 
+ \sum_{k=-\infty}^\infty e^{ik\varphi} \operatorname{res}\limits_{\tau=0}[\hat \omega_k(\tau,r)]. 
\end{align*}

Now we will find the residues of $\hat \omega_k(\tau,r)$ at $\tau=0$. 

\begin{lem} The residues of $\hat \omega_k(\tau,r)$ defined by (\ref{wksol}) at $\tau=0$  are given by
$$\operatorname{res}\limits_{\tau=0}[\hat \omega_k(\tau,r)] = {\begin{cases}0,~|k|\in \{0,1\},\\ \frac {2(|k|-1) r_0^{2|k|-2}} {r^{|k|}} \int_{r_0}^\infty s^{-|k|+1} w_k(0,s) \ds,~|k|\in \NM \setminus \{1\}.\end{cases}}
$$ 
\end{lem}

\begin{proof}Without loss of generality suppose $k \geq 0$. We will use the asymptotic form for small arguments $0 < |z| \leq \sqrt{k+1}$ \cite{BE}:
\begin{align*}
 I_{k }(z) & \sim {\frac {1}{\Gamma (k +1)}}\left({\frac {z}{2}}\right)^{k },\\K_{k }(z)&\sim {\begin{cases}-\ln \left({\dfrac {z}{2}}\right)-\gamma & {\text{if }}k =0,\\{\frac {\Gamma (k )}{2}}\left({\dfrac {2}{z}}\right)^{k }&{\text{if }}k >0,\end{cases}}
\end{align*}
where $\gamma$ is the Euler-Mascheroni constant.

For $k=0$:
$$
G_{k,1}(\tau,r) \sim  \frac {\tau r_0 \left (\ln(\sqrt \tau r /2) + \gamma\right)} {4} \int_{r_0}^\infty \left ( \ln(\sqrt \tau s /2) + \gamma  \right ) w_0(0,s) s \ds.
$$

For $k=1$:
$$
G_{k,1}(\tau,r) \sim  - \frac 1 {\tau r \left (\ln(\sqrt \tau r_0/2) + \gamma \right )  } \int_{r_0}^\infty  w_1(0,s) \ds.
$$

In both cases $k\in \{0,1\}$ the residues are equal to zero. Then only for $k>1$ $G_{k,1}(\tau,r)$ has nonzero residue:
$$
G_{k,1}(\tau,r) \sim  \frac {2(k-1) r_0^{2k-2}} {\tau r^k} \int_{r_0}^\infty s^{-k+1} w_k(0,s) \ds
$$
and
$$
\operatorname{res}\limits_{\tau=0}[G_{k,1}(\tau,r)]  = \frac {2(k-1) r_0^{2k-2}} {r^k} \int_{r_0}^\infty s^{-k+1} w_k(0,s) \ds.
$$

Now we calculate the residues of $G_{k,2}(\tau,r)$. From
\begin{align*}
K_{k }(\sqrt \tau r)I_{k }(\sqrt \tau s)&\sim {\begin{cases} -\ln \left({\dfrac {\sqrt \tau r}{2}}\right)-\gamma & {\text{if }}k =0,\\ \frac s{2kr}&{\text{if }}k >0\end{cases}}
\end{align*}
the residues of $G_{k,2}(\tau,r)$ are equal to zero. And $$\operatorname{res}\limits_{\tau=0}[\hat \omega_k(\tau,r)] = \operatorname{res}\limits_{\tau=0}[\hat G_{k,1}(\tau,r)].$$ 
\end{proof}

Since (\ref{besselrhs}) is invariant to the change $k$ to $-k$, the solution of the heat equation (\ref{maineqw}),  (\ref{initw}), (\ref{boundinfw}) with Robin condition (\ref{noslip:robin_bound}) is represented as
\begin{align}\label{heateqdirectsol}
&w(t,r,\varphi) = \nonumber 
\sum_{k=-\infty}^\infty e^{ik\varphi} \int_0^\infty  \frac{  R_{|k|,|k|-1}(\lambda, r)} 
{J_{|k|-1}(\lambda r_0)^2 + Y_{|k|-1}(\lambda r_0)^2} \\ 
&\times \left ( \int_{r_0}^\infty R_{|k|,|k|-1}(\lambda, s)  w_k(0,s) s \mathrm{d} s \right )    e^{-\lambda^2 t} \lambda \mathrm{d}\lambda \nonumber \\
&+ \sum \limits_{k=-\infty,|k| \geq 2}^\infty e^{ik\varphi} \frac {2(|k|-1) r_0^{2|k|-2}} {r^{|k|}} \int_{r_0}^\infty s^{-|k|+1} w_k(0,s) \ds,
\end{align}
where
$$
R_{k,k-1}(\lambda,s) = J_{k}(\lambda s)Y_{k-1}(\lambda r_0) - Y_{k}(\lambda s)J_{k-1}(\lambda r_0). 
$$

\begin{thm} \label{thmheateq}
Let $w_k^0(r) \sqrt r \in L_1(r_0,\infty)$ for $k \in \ZM$. Then the solution $w(t,\bx)$ of (\ref{maineqw}), (\ref{initw}),(\ref{boundinfw}), (\ref{noslip:robin_bound}) is given by (\ref{heateqdirectsol}) with $w_k(0,s) = w_k^0(s)$.
\end{thm}

\begin{proof}Using asymptotic form of the Bessel functions for large $|z|$ (\cite{BE},\cite{W}):
\begin{align*}J_{k }(z)={\sqrt {\frac {2}{\pi z}}}\left(\cos \left(z-{\frac {k \pi }{2}}-{\frac {\pi }{4}}\right)+e^{\left|\operatorname {Im} (z)\right|}\mathrm {O} \left(|z|^{-1}\right)\right),~\left|\arg z\right|<\pi ,\\Y_{k }(z)={\sqrt {\frac {2}{\pi z}}}\left(\sin \left(z-{\frac {k \pi }{2}}-{\frac {\pi }{4}}\right)+e^{\left|\operatorname {Im} (z)\right|}\mathrm {O} \left(|z|^{-1}\right)\right),~\left|\arg z\right|<\pi \end{align*}
follows that integrals in (\ref{heateqdirectsol}) are well defined and the theorem is proved. 
\end{proof}

Then instead of (\ref{int:invidentity}) we have the following invertibility relation of the associated Weber-Orr transform: 
\begin{cor}Let $f(r) \sqrt r \in L_1(r_0,\infty)$ for $k \in \NM \cup \{0\}$. Then the associated Weber-Orr transforms (\ref{int:weberorr}), (\ref{int:weberorrinv}) satisfy  almost everywhere 
\begin{align*}
f(r) - W^{-1}_{k,k-1}\left [W_{k,k-1} [f] \right ](r) = {\begin{cases} \frac {2(k-1) r_0^{2k-2}} {r^k} \int_{r_0}^\infty s^{-k+1} f(s) \ds,~k\in\NM \setminus\{1\}\\ 0,~k=0,1.\end{cases}}
\end{align*}
\end{cor}

\begin{proof}Let $w(t,\bx)$ be the solution of (\ref{maineqw}), (\ref{boundinfw}),  (\ref{noslip:robin_bound}) with initial datum $f(r)e^{ik\varphi}$. Then from (\ref{heateqdirectsol})
\begin{align*}
&w(t,r,\varphi) - 
 e^{ik\varphi} \int_0^\infty  \frac{  R_{k,k-1}(\lambda, r)} 
{J_{k-1}(\lambda r_0)^2 + Y_{k-1}(\lambda r_0)^2} \nonumber \\ 
&\times \left ( \int_{r_0}^\infty R_{k,k-1}(\lambda, s)  f(s) s \mathrm{d} s \right )    e^{-\lambda^2 t} \lambda \mathrm{d}\lambda \nonumber \\
&= \overline{I_{0,1}}(k) e^{ik\varphi} \frac {2(k-1) r_0^{2k-2}} {r^{k}} \int_{r_0}^\infty s^{-k+1} f(s) \ds,
\end{align*}
where $\overline{I_{0,1}}(k) = \begin{cases} 1,~k\in\NM \setminus\{1\} \\ 0,~k =0,1 \end{cases}$. Going to limit $t \to 0$ we will obtain the required relation.
\end{proof}

Finally, we have
\begin{thm} Let vector field $\bv_0(\bx)$ satisfies (\ref{freediv}), (\ref{bound}), (\ref{boundinf}), (\ref{zerocirculation}), $\curl \bv_0(\bx)$ $ \in L_1(B_{r_0})$, and its Fourier series as well as ones for vorticity converges and coefficients $w_k^0(r)$ satisfy $w_k^0(r) \sqrt r \in L_1(r_0,\infty)$, $k \in \ZM$. Then the solution $w(t,\bx)$ of (\ref{maineqw})-(\ref{boundinfw}) is expressed as
\begin{align*}
w(t,r,\varphi) = \nonumber 
\sum_{k=-\infty}^\infty e^{ik\varphi} \int_0^\infty  \frac{  R_{|k|,|k|-1}(\lambda, r)} 
{J_{|k|-1}(\lambda r_0)^2 + Y_{|k|-1}(\lambda r_0)^2} \\ \times \left ( \int_{r_0}^\infty R_{|k|,|k|-1}(\lambda, s)  w_k^0(s) s \mathrm{d} s \right )  e^{-\lambda^2 t} \lambda \mathrm{d}\lambda.
\end{align*}
\end{thm}

\begin{proof} From $\curl \bv_0(\bx)$ $ \in L_1(B_{r_0})$ follows
$w_0^0(r) r \in L_1(r_0,\infty)$ and (\ref{BS:cond}) is satisfied for $k = 0$. From $w_k^0(r) \sqrt r \in L_1(r_0,\infty)$ follows (\ref{BS:cond}) for $k \neq 0$. Then from Theorems \ref{thmrobincond}, \ref{thmheateq} we have the formula (\ref{heateqdirectsol}) for $w(t,\bx)$. And in virtue of (\ref{noslipcondintegral}) the last term in (\ref{heateqdirectsol}) vanishes. 
\end{proof}

~\\~ 
A.\,V.~Gorshkov\\
Lomonosov Moscow State University,\\ 
Leninskie Gory, Moscow, 119991, \\Russian Federation\\
alexey.gorshkov.msu@gmail.com\\

\end{document}